\documentclass[12pt,twoside]{article}
\usepackage{amsmath,amsfonts,amssymb}
\usepackage{graphicx}
\usepackage{eucal}
\usepackage{mathrsfs}
\usepackage{theorem}
\usepackage{pifont}
\usepackage {epsfig}
\usepackage {graphicx}
\usepackage{dsfont} 
\usepackage{url}

\newcommand{\Real}{\mathbb{R}}
\newcommand{\Integer}{\mathbb{Z}}
\newcommand{\Natural}{\mathbb{N}}
\newcommand{\Complex}{\mathbb{C}}

\newcommand{\lpz}{l^p(\mathbb{Z})}

\newcommand{\N}{\mathbb{N}}

\newcommand{\todo}[1]{{\sffamily To do:}}
\newcommand{\NN}{\mathcal{N}}
\newcommand{\vectornorm}[1]{\|#1\|}
\newtheorem{theorem}{Theorem}[section]
\newtheorem{proposition}[theorem]{Proposition}

\newtheorem {lemma}[theorem]{Lemma}
\theorembodyfont{\rm}
\newtheorem{examples}[theorem]{Examples}
\newtheorem{example}[theorem]{Example}
\newtheorem{definition}[theorem]{Definition}
\newtheorem{remark}[theorem]{Remark}
\newenvironment{proof}{{\flushleft \emph{Proof}:}}{\hfill\ding{110}\\}

{\theorembodyfont{\rmfamily}

}
\numberwithin{equation}{section}
\begin{document}

\title{On the distribution of the discrete spectrum of nuclearly perturbed  operators in Banach spaces}
\date{}
\maketitle
Michael Demuth,\\
Institute of Mathematics,\\Technical University of Clausthal,\\ 38678 Clausthal-Zellerfeld\\
\url{michael.demuth@tu-clausthal.de}\\
\newline
Franz Hanauska,\\
Institute of Mathematics,\\Technical University of Clausthal,\\ 38678 Clausthal-Zellerfeld\\
\url{franz.hanauska@tu-clausthal.de}\\



{\bf Abstract:}
Let $Z_0$ be a bounded operator in a Banach space $X$ with purely essential spectrum and $K$ a nuclear operator in $X$. Using methods of complex analysis we study the discrete spectrum of $Z_0+K$ and derive a Lieb-Thirring type inequality. We obtain estimates for the number of eigenvalues in certain regions of the complex plane and an estimate for the asymptotics of the eigenvalues approaching to the essential spectrum of $Z_0$.\\
\newline 
{\bf Mathematics Subject Classifictaion (2010):} 47A75, 47A10, 47A55, 47B10.


\newpage

\section{Introduction}
In the present article we analyze the discrete  spectrum of a linear, bounded operator $Z=Z_0+K$, where $Z_0$ is a bounded operator with purely essential spectrum and $K$ is a nuclear perturbation. If $X$ is a Hilbert space and if $K$ is in some Neumann-Schatten class this problem was well studied during the last years, see for instance for general non-selfadjoint operators Demuth, Hansmann and Katriel \cite{demuth}, for non-selfadjoint perturbations of selfadjoint operators Hansmann \cite{hansmann2012}, for Schr\"odinger operators Frank \cite{frank}, Laptev and Safronov \cite{laptev}, Safronov \cite{safronov}, Hansmann \cite{hansmann2011} and for Jacobi operators Borichev, Golinskii and Kupin \cite{borichev}, Favorov and Golinskii \cite{favorov}, Golinskii and Kupin \cite{golinskii} and Hansmann and Katriel \cite{hansmannkatriel}\\
It turns out that we can generalize the theory known for Hilbert spaces if we can prove the so-called nuclear determinant, $\det\big(\mathds{1}-K(z\mathds{1}-Z_0)^{-1}\big)$ to be an analytic function on the resolvent set of $Z_0$.\\
In Section 2 we explain the method of the proof which uses substantially the behaviour of the zeros of holomorphic functions defined in the open unit disc. \\
In Section 3 the main result is the holomorphy of the nuclear determinant. \\
In the last section we apply the results to the discrete Laplacian $\Delta_p$ in $l^p(\Integer)$. It turns out that 
\begin{align*}
\sum_{z\in \sigma_{disc}(\Delta_p+K)} \frac{\text{dist}\big(z,\sigma_{ess}(\Delta_p+K)\big)^{3+\tau}}{|z^2-4|}\leq c(\tau)\|K\|_{\NN}^2,
\end{align*}
with some $\tau>0$, a constant $c(\tau)>0$ and where $\|\cdot\|_{\NN}$ denotes the norm in the space of nuclear operators in $X$. This inequality can be used to estimate the number of the eigenvalues in certain parts of the complex plane or to estimate the possible asymptotics if the eigenvalues approach to $\sigma_{ess}(\Delta_p+ K)=[-2,2]$ (see Chapter 5).

\section{Objective and Motivation}
Let $X$ be a complex Banach space and $Z_0$ a bounded operator on $X$ with purely essential spectrum ($\sigma_{ess}(Z):=\{\lambda\in\Complex: \lambda-Z$ is not a Fredholm operator$\}$, where an operator $A$ is Fredholm if $A$ has closed range and both, the kernel and the cokernel of $A$ are finite dimensional) which is equal to an intervall, i.e. $\sigma(Z_0)=\sigma_{ess}(Z_0)=[a,b]$. We denote by $(\lambda\mathds{1}-Z_0)^{-1}=:R_{Z_0}(\lambda), \lambda \in\rho(Z_0):=(\sigma(Z_0))^c$ (resolvent set) the resolvent of $Z_0$. \\
We perturb $Z_0$ by a nuclear operator $K$ and define
\begin{align*}
Z:=Z_0+K.
\end{align*}
We are interested in the distribution of the discrete spectrum ($\sigma_{disc}(Z):=\{\lambda\in\Complex: \lambda$ is a discrete eigenvalue of $Z\}$, where an eigenvalue is discrete if it is isolated and its corresponding Riesz projection is of finite rank) of $Z$. For the sake of completeness we repeat here the definition of nuclear operators.
\begin{definition}\label{2.1}
Let $K$ be a compact operator in $\mathcal{B}(X)$ (the space of linear bounded operators). $K$ is called {\bf nuclear} if there are sequences (not necessarily unique) $\{f_n\}\subseteq X$, $\{\phi_n\}\subseteq X^*$ (the dual of $X$) such that $Kf$ can be represented by 
\begin{align*}
Kf = \sum_{n=1}^\infty \langle \phi_n, f \rangle f_n 
\end{align*}
for all $f \in X$ and

\begin{align*}
\sum_{n=1}^\infty \|\phi_n\|_{X^*}\|f_n\|_X < \infty.
\end{align*}
We denote this class by $\mathcal{N}(X)$
\end{definition}

In $\mathcal{N}(X)$ a norm can be defined by 
\begin{align*}
\|K\|_{\NN}:=\inf\{\sum_{n=1}^\infty \|\phi_n\|_{X^*}\|f_n\|_X \, : \, Kf=\sum_{n=1}^\infty \langle\phi_n, f\rangle f_n\text{ for all }f\in X\}.
\end{align*}
With this norm $\NN(X)$ becomes a Banach ideal (see Pietsch \cite{pietsch}, p. 64).

\begin{examples}\label{beispiel}
\quad (a) Let $\{e_k\}_{k\in\Integer}$ be the standard basis in $l^p(\Integer)$ with $1\leq p\leq \infty$. Denote by $\phi_m$ the sequence $\phi_m=\{a_{mj}\}_{j\in\Integer}\in l^q(\Integer)$ ($\frac{1}{p}+\frac{1}{q}=1$). Assuming $\{\|\phi_m\|_q\}_{m\in\mathbb{Z} }\in l^1(\Integer)$, then the operator $K:\lpz\rightarrow \lpz$ defined by $Kf:=\sum_{m\in\Integer}\langle \phi_m,f\rangle e_m$ is nuclear. The corresponding infinite matrix is given by $(a_{mj})_{m,j\in\Integer}$.\\
We can conclude that every diagonal operator which is defined by an infinite matrix
$\text{diag}(...,d_{-1},d_0,d_1,...)$ is nuclear if $\{d_n\}_{n\in\Integer}\in l^1(\Integer)$.\\
\newline
(b) Every integral operator 
\begin{align*}
K:C([\alpha,\beta])\rightarrow C([\alpha,\beta]),\, (Kf)(t):=\int_\alpha^\beta k(t,s)f(s)ds
\end{align*}
 with continuous kernel $k$ is nuclear and $\|K\|_{\mathcal{N}}=\int_\alpha^\beta \max_t |k(t,s)|ds$ (see Gohberg, Goldberg and Krupnik \cite{gohberg}, Chapter 2 Theorem 2.2).
\end{examples}

\begin{remark}
If $X$ is a Hilbert space $\NN(X)$ coincides with the ideal of trace class operators. In this case we know that the eigenvalues are summable. However, there are Banach spaces and nuclear operators with non summable eigenvalues (see e.g. Gohberg, Goldberg and Krupnik \cite{gohberg} p. 102).\\

In general one has the following estimate:\\
Let $\{\lambda_n(K)\}$ be the eigenvalues of the nuclear operator $K$, then
\begin{align}
\sum_{n=1}^\infty |\lambda_n(K)|^2 \leq \|K\|_{\NN}^2,\label{absch.nukl}
\end{align}
(see e.g. Pietsch \cite{pietsch}, p. 160).
\end{remark}

\begin{example}\label{nuclearnothibert}
If $X_1$ and $X_2$ are compatible Banach spaces and if $K_1$ and $K_2$ are consistent compact operators acting in $X_1$ and $X_2$ then (see \cite{davies} p. 107).
\begin{align*}
\sigma(K_1)=\sigma(K_2).
\end{align*}
We know that for $1\leq p_1,p_2<\infty$ the spaces $l^{p_1}(\mathbb{N})$ and $l^{p_2}(\mathbb{N})$ are compatible.\\ 
Now let $K_1$ be an operator on $l^1(\mathbb{N})$ and $K_2$ be an operator on $l^2(\mathbb{N})$ and let $K_1$ and $K_2$ be consistent. If the eigenvalues of $K_1$ are square summable the same is true for $K_2$.
Now let $K_2$ be an operator defined on $l^2(\mathbb{N})$ which is consistent to a nuclear operator $K_1$ defined on $l^1(\mathbb{N})$. Then $K_2$ is not automatically a Hilbert-Schmidt operator or a trace-class operator.\\
To check this we define the infinite matrix
\begin{align*}
(a_{km})_{k,m\in\mathbb{N}}:=
\begin{pmatrix}
2^{-1} & 2^{-1} & 2^{-1} & \dots\\
2^{-2} & 2^{-2} & 2^{-2} & \dots\\
2^{-3} & 2^{-3} & 2^{-3} & \dots\\
\vdots & \vdots & \vdots & 
\end{pmatrix}
\end{align*}
and define with this matrix the operators $K_1$ and $K_2$.\\
For $K_1$ the nuclear norm is $\|K_1\|_{\mathcal{N}}=\sum_{k=1}^\infty \sup_m |a_{km}|$ (see \cite{gohberg}, Chapter V Theorem 2.1). So we have
\begin{align*}
\|K_1\|_{\mathcal{N}}=\sum_{k=1}^\infty 2^{-k}=1
\end{align*}
such that $K_1$ is in fact a nuclear operator.\\
$K_2$ is a Hilbert-Schmidt operator on $l^2(\mathbb{N})$ iff the sum $\sum_{j=1}^\infty\|K_2e_j\|_2$ is finite, where $(e_j)$ is the orthonormal standard basis in $l^2(\mathbb{N})$ (see \cite{gohberg}, Chapter IV Theorem 7.1).In the present example
\begin{align*}
\sum_{j=1}^\infty \|K_2e_j\|_2^2=\sum_{j=1}^\infty \|(2^{-k})\|_2^2=\infty,
\end{align*}
that means $K_2$ is not a Hilbert-Schmidt operator and hence not a trace class operator.
\end{example}

Because every nuclear operator $K$ is compact $\sigma_{ess}(Z)=\sigma_{ess}(Z_0)$ (see \cite{gohberg1} Chapter XI Theorem 4.2) and the spectrum of $Z$ is the disjoint union of $\sigma_{ess}(Z)$ and $\sigma_{disc}(Z)$.\\
\newline
We are interested in estimates of the form 
\begin{align*}
\sum_{\lambda\in\sigma_{disc}(Z)}\frac{(\text{dist}(\lambda,[a,b]))^\alpha}{|b-\lambda|^\beta|a-\lambda|^\beta}\leq C(\alpha,\beta,a,b)\|K\|_{\NN}^\gamma
\end{align*}
with positive exponents $\alpha,\beta,\gamma\in\Real$.\\
\newline
Instead of studying $\sigma_{disc}(Z)$ directly we define a holomorphic function in $\mathbb{C}\setminus[a,b]$ the zeros of which coincide with $\{\lambda_n(Z)\}$. Then we study the behaviour of the zeros of holomorphic functions in the unit disc. Finally we transform the problem back and can analyze $\sigma_{disc}(Z)$. The function we have in mind is the determinant of $\mathds{1}-KR_{Z_0}(\lambda)$.

\begin{definition}\label{def2.3}
The {\bf nuclear determinant} (or also called regularized determinant see \cite{gohberg} Chapter IX) of a nuclear  operator $K$ in $X$ is given by 
\begin{align*}
\det(\mathds{1}-K):=\prod_{n=1}^\infty \big(1-\lambda_n(K)\big)\exp\big(\lambda_n(K)\big.)
\end{align*}
where $\{\lambda_n(K)\}$ are again the eigenvalues of $K$.
\end{definition}
The determinant has some important properties used in this article which we summarize here.

\begin{lemma}\label{2.5}
Let $K \in \NN(X)$. Then 
\begin{enumerate}
\item[(i)] $|\det(\mathds{1}-K)|\leq \exp\big(\frac{1}{2}\|K\|_{\NN}^2\big)$, which implies the existence of the determinant.
\item[(ii)] $\det(\mathds{1}-K)=0$ iff $\,\lambda_n(K)=1$ for some $n\in\Natural$.
\item[(iii)] $\det(\mathds{1}-K)=0$ iff $\,\mathds{1}-K$ is not invertible.
\end{enumerate}
\end{lemma}
\begin{proof}
(ii) and (iii) are obvious. (i) follows by the inequality
\begin{align*}
|(1-z)\exp(z)| \leq \exp\bigg(\frac{1}{2}|z|^2\bigg)
\end{align*}
which holds for all $z\in\mathbb{C}$ (see for instance Nevanlinna  \cite{nevanlinna} p. 225).\\
Hence we obtain, using (\ref{absch.nukl}),
\begin{align*}
|\det(\mathds{1}-K)| &\leq \prod_{n=1}^\infty \exp\bigg(\frac{1}{2}|\lambda_n(K)|^2\bigg)\\
&=\exp\bigg(\sum_{n=1}^\infty \frac{1}{2}|\lambda_n(K)|^2\bigg)\\
&\leq\exp\bigg(\frac{1}{2}\|K\|_{\NN}^2\bigg).
\end{align*}
\end{proof}
Let $Z_0$ be as mentioned above and $Z=Z_0+K$, $K\in\NN(X)$, such that $\sigma_{ess}(Z)=\sigma_{ess}(Z_0)=[a,b]$.\\
Take $\lambda_0\in \rho(Z_0)$. Then 
\begin{align*}
(\lambda_0\mathds{1}-Z)R_{Z_0}(\lambda_0) = \mathds{1}-KR_{Z_0}(\lambda_0).
\end{align*}
The operator $\mathds{1}-KR_{Z_0}(\lambda_0)$ is not invertible iff $\lambda_0 \in \sigma_{disc}(Z)$. Because $\NN(X)$ is an ideal $KR_{Z_0}(\lambda_0)\in\NN(X)$. Therefore the determinant 
\begin{align*}
\det\big(\mathds{1}-KR_{Z_0}(\lambda)\big)
\end{align*}
is well defined for any $\lambda \in \rho(Z_0)$. Denote by 
\begin{align*}
d(\cdot):=\det\big(\mathds{1}-KR_{Z_0}(\cdot)\big),
\end{align*}
i.e. the map $\rho(Z_0)\ni\lambda \mapsto \det\big(\mathds{1}-KR_{Z_0}(\lambda)\big)$. The complex number $\lambda_0$ is a zero of $d$ iff $\lambda_0\in \sigma_{disc}(Z)$. Denoting the zero set of $d$ by $\mathcal{Z}(d)$ it follows that $\mathcal{Z}(d)=\sigma_{disc}(Z)$.\\
\newline
\begin{remark}
It is possible to extend the domain of $d$ to $\rho(Z_0)\cup \{\infty\}$ by setting $d(\infty):=1$. This definition makes sense, since $\lim_{\lambda\rightarrow \infty}KR_{Z_0}(\lambda)=0$ (see e.g. Kato \cite{kato}, p. 176) and $\det(\mathds{1}-0)=1$.
\end{remark}
Thus we are able to analyze $\sigma_{disc}(Z)$ by studying the zeros of the function $d$ defined on $\mathbb{C}\setminus[a,b]$.\\
\newline
We will follow the strategy used in \cite{demuth}, that is based on Jensen's identity (see Rudin \cite{rudin}, p. 307) for the zeros of holomorphic functions in the open unit disc $\mathbb{D}$. Let $\phi$ be the conformal map from $\mathbb{D}\setminus \{0\}$ to $\mathbb{C}\setminus[a,b]$ given by 
\begin{align}\label{8.1}
\phi(w)=\frac{b-a}{4}\big(w+w^{-1}+2\big)+a.
\end{align}
Then the new function $h$, given by
\begin{align*}
h(w):=
\begin{cases}
&(d\circ \phi)(w),\quad w\in\mathbb{D}\setminus\{0\}\\
&1,\quad w =0
\end{cases}
\end{align*}
is defined on $\mathbb{D}$. \\
Let $\mathcal{Z}(h)$ be the set of zeros of $h$. If we can show that $h$ is a holomorphic function and since $|h(0)|=1$
\begin{align}\label{8.2}
\sum_{w\in \mathcal{Z}(h), |w|\leq r } \log\big|\frac{r}{w}\big| =\frac{1}{2\pi}\int_0^{2\pi}\log|h(re^{i\theta})|d \theta
\end{align}
with $0<r<1$.
If $h$ is a holomorphic function and if there is a proper estimate for $\log|h(re^{i\theta})|$ such that the left hand side in (\ref{8.2}) gives an effective sum over the zeros of $h$, then we can derive from this sum a new sum over the discrete spectrum of $Z$.\\

For instance if, in the simplest case,
\begin{align}\label{9.1}
\log|h(w)|\leq \frac{C_0}{\big(1-|w|\big)^\alpha}, \quad w \in \mathbb{D},
\end{align}
with $\alpha>0$ and $C_0$ a positive constant, then 
\begin{align}\label{9.2}
\sum_{w\in\mathcal{Z}(h)}\big(1-|w|\big)^{\alpha+\tau+1} \leq C(\alpha,\tau)C_0
\end{align}
for any $\tau>0$ (see \cite{demuth}, Theorem 3.3.1).\\
(\ref{9.2}) implies 
\begin{align}
\sum_{\lambda\in\sigma_{disc}(Z)}\big(1-|\phi^{-1}(\lambda)|\big)^{\alpha+\tau+1} \leq C(\alpha,\tau)C_0.
\end{align}
Obviously by Lemma \ref{2.5} we know, since $\NN(X)$ is a Banach ideal and $R_{Z_0}(\phi(w))$ is bounded, that
\begin{align}\label{9.4}
\log|h(w)|\leq \frac{1}{2}\|K\|_{\NN}^2 \|R_{Z_0}(\phi(w))\|^2.
\end{align}
If we can estimate the resolvent in a similar way that finally (see Golinskii and Kupin \cite{golinskii}, Hansmann and Katriel \cite{hansmannkatriel})
\begin{align}\label{9.5}
\log|h(w)|\leq C_0 \frac{|w|^2}{|w-1|^2|w+1|^2(1-|w|)^2}
\end{align}
then for any $\tau>0$
\begin{align*}
\sum_{\lambda\in\sigma_{disc}(Z)}\frac{(1-|\phi^{-1}(\lambda)|)^{3+\tau}}{|\phi_1^{-1}(\lambda)|^{1+\tau}}|(\phi_1^{-1}(\lambda))^2-1|^{1+\tau}\leq c(\tau)C_0 .
\end{align*}
From this estimate we are able to derive
\begin{align}\label{10.1}
\sum_{\lambda\in\sigma_{disc}(Z)}\frac{\big(\text{dist}(\lambda,[a,b])\big)^{3+\tau}}{|\lambda-a||\lambda-b|}\leq c(\tau)\cdot C_0
\end{align}
(see \cite{demuth}, proof of Theorem 4.2.2).\\
Summarizing this procedure the problems which we have to solve in this context are:
\begin{enumerate}
\item[1)] Prove that the determinant $d(\cdot)=\det\big(\mathds{1}-KR_{Z_0}(\cdot)\big)$ is a holomorphic function in $\mathbb{C}\setminus[a,b]$.
\item[2)] Find an example and an estimate like (\ref{9.5}) such that (\ref{10.1}) will be true.
\end{enumerate}
Both has already been done in Hilbert spaces and for Schatten class perturbations. However here we consider nuclear perturbations in Banach spaces.\\
\newline
One part of this article is the proof of the holomorphy of $d$. This is the content of the next section, which we have not found in the literature.\\
\newline
In the final sections we consider nuclear perturbations of the discrete Laplacian in $l^p(\Integer)$ and certain perturbations by integral operators in $C[\alpha,\beta]$. It turns out that an estimate like (\ref{10.1}) can be verified.

\section{Holomorphy of nuclear determinants}
Let $\lambda \mapsto K(\lambda)$ be a holomorphic family of operators in $\NN(X)$ with $\lambda$ in a domain $\Omega\subseteq \mathbb{C}$. That means for every $\lambda_0\in\Omega$  there is a sequence of nuclear operators $\{K_{n,\lambda_0}\}$ such that 
\begin{align*}
K(\lambda)=\sum_{n=0}^\infty (\lambda-\lambda_0)^n K_{n,\lambda_0}
\end{align*}
and 
\begin{align*}
\sum_{n=0}^\infty|\lambda-\lambda_0|^n\|K_{n,\lambda_0}\|_{\NN} <\infty
\end{align*}
for all $\lambda\in B_{R_{\lambda_0}}:=\{\lambda\,:\,|\lambda-\lambda_0|<R_{\lambda_0}\}$, where $R_{\lambda_0}$ is the radius of convergence depending only on $\lambda_0$. \\
\newline
We intend to show that the map (see Definition \ref{def2.3})
\begin{align*}
\lambda \mapsto \det\big(\mathds{1}-K(\lambda)\big)
\end{align*}
is holomorphic in $\Omega$ (Pietsch \cite{pietsch} showed this for multiplicative determinants. In our situation the determinant is not multiplicative.).

\begin{remark}
Let $\mathcal{F}(X):=\{F\in\mathcal{S}_\infty (X): \dim(\text{ran}(F))<\infty\}$ ($\mathcal{S}_\infty(X)$ denotes the ideal of compact operators on $X$). Recall that $\mathcal{F}(X)$ is dense in $\big(\NN(X),\|\cdot\|_{\NN}\big)$. This fact will be used in the following.
\end{remark}

\begin{lemma}\label{lemma3.1}
Let $\lambda \mapsto K(\lambda)$ be an analytic nuclear operator valued function on a domain $\Omega \subseteq \Complex$.\\
Then for every $\lambda_0\in\Omega$ there is a sequence of analytic mappings $\lambda\mapsto K_{n,\lambda_0}(\lambda) \in \mathcal{F}(X)$ such that $K_{n,\lambda_0}(\cdot)$ converges locally uniformly to $K(\cdot)$ in $B_{R_{\lambda_0}}$ (with respect to the nuclear norm). Moreover, for every $n\in\Natural$ there is a linear subspace $M_{n,\lambda_0}\subseteq X$ of finite dimension such that 
\begin{align*}
K_{n,\lambda_0}(\lambda)X = K_{n,\lambda_0}(\lambda)M_{n,\lambda_0}\subseteq M_{n,\lambda_0}
\end{align*}
for all $\lambda \in B_{R_{\lambda_0}}$.
\end{lemma}

\begin{proof}
It is sufficient to show this for $\lambda_0=0$. Then $K(\lambda)$ has the representation

\begin{align*}
K(\lambda)=\sum_{j=0}^\infty \lambda ^j K_j, \quad K_j \in \mathcal{N}(X).
\end{align*}

Let $\epsilon >0$. Then there is an $n_0\in\N$ with

\begin{align}
\sum_{j=n_0+1}^\infty |\lambda|^j\vectornorm{K_j}_\mathcal{N} < \frac{\epsilon}{2}, \text{ for all } |\lambda |\leq s < R_{\lambda_0}.
\end{align}

For every $j\in\{0,...,n_0\}$ we choose a finite rank operator $K_j^{(n_0)}$ with

\begin{align}
\vectornorm{K_j-K_j^{(n_0)}}_\mathcal{N} < \frac{\epsilon}{2(n_0+1)(R_{\lambda_0}+1)^{n_0}},
\end{align}

and we define 

\begin{align}
K_{n_0}(\lambda)=\sum_{j=0}^{n_0}\lambda^j K_j^{(n_0)}
\end{align}

The function $\lambda \mapsto K_{n_0}(\lambda)$ is analytic.\\
Thus we get

\begin{align}
&\vectornorm{K(\lambda)-K_{n_0}(\lambda)}_\mathcal{N}=\vectornorm{\sum_{j=n_0+1}^\infty \lambda^j K_j+\sum_{j=0}^{n_0}\lambda^j(K_j-K_j^{(n_0)})}_\mathcal{N}\\
&\leq \sum_{j=n_0+1}^\infty |\lambda|^j \vectornorm{K_j}_\mathcal{N}+\sum_{j=0}^{n_0}|\lambda|^j \vectornorm{K_j-K_j^{(n_0)}}_\mathcal{N}\\
&\leq \frac{\epsilon}{2} +
\sum_{j=0}^{n_0}(R_{\lambda_0}+1)^{n_0}\frac{\epsilon}{2(n_0+1)(R_{\lambda_0}+1)^{n_0}}\leq\epsilon.
\end{align}

That implies the first assertion.\\
For the second assertion there is for every $j\in\{0,...,n_0\}$ a finite dimensional linear space $M_j^{n_0}\subseteq X$, with (see Gohberg, Goldberg and Krupnik \cite{gohberg}, chapter I Lemma 1.1)
\begin{align*}
K_j^{(n_0)}X=K_j^{(n_0)}M_j^{n_0}\subseteq M_j^{n_0}.
\end{align*}
Defining the finite dimensional subspace $M_{n_0}:=\sum_{j=0}^{n_0}M_j^{n_0}$ we have
\begin{align*}
K_{n_0}(\lambda)X\subseteq K_{n_0}(\lambda)M_{n_0}\subseteq M_{n_0},\text{ for all } \lambda \in B_{R_{\lambda_0}}\text{ (even for all }\lambda\in\Complex).
\end{align*}
\end{proof}
\begin{lemma}\label{lemma3.2}
Let $\lambda \mapsto F(\lambda)\in \mathcal{F}(X)$ be analytic on a domain $\Omega \subseteq \mathbb{C}$. If there is a finite dimensional subspace $M_{\lambda_0}$ for every  $B_{R_{\lambda_0}}\subseteq \Omega$ with ran$(F(\lambda))=F(\lambda)M_{\lambda_0}\subseteq M_{\lambda_0}$ for all $\lambda \in B_{R_{\lambda_0}}$ then

\begin{align*}
d:\Omega\rightarrow \mathbb{C}, \quad d(\lambda):=\det (1-F(\lambda))
\end{align*}

defines a holomorphic function in $\Omega$.

\end{lemma}

\begin{proof}
By assumption for every disc $B_{R_{\lambda_0}}$ there is a finite-dimensional subspaces $M_{\lambda_0}$ with

\begin{align*}
F(\lambda)M_{\lambda_0}=\text{ran}(F(\lambda))\subseteq M_{\lambda_0},\text{ for all } \lambda \in B_{R_{\lambda_0}}.
\end{align*}

Then for every $\lambda\in B_{R_{\lambda_0}}$ we get for $F(\lambda)|_{M_{\lambda_0}}:M_{\lambda_0}\rightarrow M_{\lambda_0}$

\begin{align}
F(\lambda)|_{M_{\lambda_0}} = \sum_{k=0}^\infty (\lambda^k-\lambda_0) F_{k,\lambda_0}|_{M_{\lambda_0}} \label{3.2.1}
\end{align}
and
\begin{align*}
\sigma(F(\lambda))\setminus\{0\}=\sigma(F(\lambda)|_{M_{\lambda_0}})\setminus\{0\} \text{ (finite sets)}.
\end{align*}

Because $M_{\lambda_0}$ is a finite-dimensional linear space, it is possible to consider $M_{\lambda_0}$ as a Hilbert-space, and $F(\lambda)|_{M_{\lambda_0}}$ as a Hilbert-Schmidt operator ($F(\lambda)\in \mathcal{S}_2({M_{\lambda_0}})$ for every $\lambda \in B_{R_{\lambda_0}}$). With (\ref{3.2.1}) we have the information that $F(\lambda)|_{M_{\lambda_0}}$ is analytic.\\
For a Hilbert-Schmidt operator $A$ in a Hilbert space the Hilbert-Schmidt determinant is 
\begin{align*}
\text{det}_{HS}(\mathds{1}-A):=\prod_{j=1}^\infty \big(1-\lambda_j(A)\big)\exp\big(\lambda_j(A)\big).
\end{align*}
That means in this situation the nuclear determinant in Definition \ref{def2.3} coincides with $\det_{HS}$. For $F(\lambda)|_{M_{\lambda_0}}$, $\lambda\in B_{R_{\lambda_0}}$, we have $\det_{HS}\big(\mathds{1}-F(\lambda)|_{M_{\lambda_0}}\big)=\det \big(\mathds{1}-F(\lambda)\big).$ \\
The function $\lambda\mapsto\det_{HS}(\big(\mathds{1}-F(\lambda)|_{M_{\lambda_0}}\big)$ is analytic in $B_{R_{\lambda_0}}$ which is proved e.g. in Simon \cite{simon} p.254 and p.261.\\
Hence $d(\cdot)=\det\big(\mathds{1}-F(\cdot)\big)$ is holomorphic in $B_{R_{\lambda_0}}$ and so also in $\Omega$.
\end{proof}

\begin{remark}
With the previous argumentation we know, that $\lambda\mapsto \det\big(\mathds{1}-(A+\lambda B)\big)$ is an entire-function for all $A,B\in\mathcal{F}(X)$. If there is in addition a continuous monoton non-decreasing function $g$ on $[0,\infty)$ such that
\begin{align}
|\text{det}(1-F)|\leq g(\|F\|_\mathcal{N}) \text{ for all }F\in\mathcal{F}(X) \label{remark3.4}
\end{align}
then (see Gohberg, Goldberg and Krupnik \cite{gohberg}, Chapter II Theorem 4.1) for any $A,B\in\mathcal{F}(X)$ 
\begin{align}
|\det(\mathds{1}-A)-\det(\mathds{1}-B)| \leq \vectornorm{A-B}_\mathcal{N}g(\|A\|_\mathcal{N}+\|B\|_\mathcal{N}+1). \label{1.4}
\end{align}
By Lemma \ref{2.5} (i) $g(t):=\exp\big(\frac{1}{2}t^2\big)$ satisfies (\ref{remark3.4}) and hence we can derive from (\ref{1.4}) the estimate
\begin{align*}
|\det(\mathds{1}-A)-\det(\mathds{1}-B)|\leq \vectornorm{A-B}_\mathcal{N}\exp\bigg(\frac{1}{2}(\vectornorm{A}_\mathcal{N}+\vectornorm{B}_\mathcal{N}+1)^2\bigg)
\end{align*}
for all $A,B\in\mathcal{F}(X)$.
\end{remark}
Now we want to extend this inequality from $\mathcal{F}(X)$ to $\mathcal{N}(X)$.

\begin{lemma}\label{lemma3.4}
For all $A,B\in\mathcal{N}(X)$ holds

\begin{align*}
|\det(\mathds{1}-A)-\det(\mathds{1}-B)|\leq \vectornorm{A-B}_\mathcal{N}\exp\bigg(\frac{1}{2}(\vectornorm{A}_\mathcal{N}+\vectornorm{B}_\mathcal{N}+1)^2\bigg).
\end{align*}

\end{lemma}

\begin{proof}

Let $A,B\in\mathcal{N}(X)$. Then there are sequences $(f_n),(g_n)\subseteq X, (\phi_n),(\psi_n)\subseteq X^*$, with
\begin{align*}
\begin{matrix}
A\cdot=\sum_{k=1}^\infty \langle\phi_k,\cdot\rangle f_k,  &B\cdot=\sum_{k=1}^\infty \langle\psi_k,\cdot\rangle g_k\\
\sum_{k=1}^\infty \vectornorm{\phi_k}_{X^*}\vectornorm{f_k}_X < \infty,  &\sum_{k=1}^\infty \vectornorm{\psi_k}_{X^*}\vectornorm{g_k}_X < \infty
\end{matrix}
\end{align*}

The infinite matrices $\tilde{A}:=(a_{ij})_{i,j\in\N}$, $\tilde{B}:=(b_{ij})_{i,j\in\N}$, with the entries
\begin{align*}
\begin{matrix}
a_{ij}:=\frac{\vectornorm{f_i}_X^\frac{1}{2}}{\vectornorm{\phi_i}_{X^*}^\frac{1}{2}}\langle \phi_i,f_j\rangle \frac{\vectornorm{\phi_j}_{X^*}^\frac{1}{2}}{\vectornorm{f_j}_{X}^\frac{1}{2}}\\
b_{ij}:=\frac{\vectornorm{g_i}_X^\frac{1}{2}}{\vectornorm{\psi_i}_{X^*}^\frac{1}{2}}\langle \psi_i,g_j\rangle \frac{\vectornorm{\psi_j}_{X^*}^\frac{1}{2}}{\vectornorm{g_j}_{X}^\frac{1}{2}}
\end{matrix}
\end{align*}

are linear operators in $\mathcal{S}_2(l^2(\N))$ and for their spectra holds (Gohberg, Goldberg and Krupnik \cite{gohberg} p.106-107)

\begin{align}
\sigma(A)=\sigma(\tilde{A}), \quad \sigma(B)=\sigma(\tilde{B}). \label{3.4.1}
\end{align}
By definition we obtain for the nuclear determinant and the Hilbert-Schmidt determinant

\begin{align*}
\det(\mathds{1}-A)=\text{det}_{HS}(\mathds{1}-\tilde{A}),\, \det(\mathds{1}-B)=\text{det}_{HS}(\mathds{1}-\tilde{B}).
\end{align*}

Now we define the finite rank operators 

\begin{align*}
\begin{matrix}
A_n=\sum_{k=1}^n \phi_k\otimes f_k,  &B_n=\sum_{k=1}^n \psi_k\otimes g_k
\end{matrix}
\end{align*}

and the finite rank matrices

\begin{align*}
\begin{matrix}
\tilde{A}_n:=(\tilde{a}_{ij})_{i,j\in\N}, &\tilde{B}_n:=(\tilde{b}_{ij})_{i,j\in\N}
\end{matrix}
\end{align*}
where $\tilde{a}_{ij}=a_{ij}, \tilde{b}_{ij}=b_{ij}$ for $1\leq i,j\leq n$, else $\tilde{a}_{ij}=\tilde{b}_{ij}=0$.
Then we have
\begin{align*}
&\vectornorm{A_n-A}_\mathcal{N}\rightarrow 0,&\, \vectornorm{B_n-B}_\mathcal{N}\rightarrow 0,\\
&\vectornorm{\tilde{A}_n-\tilde{A}}_{HS}\rightarrow 0, &\,\vectornorm{\tilde{B}_n-\tilde{B}}_{HS}\rightarrow 0 
\end{align*}
as $n\rightarrow \infty$. Again the determinants coincide, i.e. 
\begin{align*}
\text{det}_{HS}(\mathds{1}-\tilde{A_n})=\det(\mathds{1}-A_n).
\end{align*}
The Hilbert-Schmidt determinant is continuous with respect to the Hibert-Schmidt norm.\\
That implies
\begin{align*}
&\lim_{n\rightarrow\infty}\det(\mathds{1}-A_n)=\lim_{n\rightarrow\infty}\text{det}_{HS}(\mathds{1}-\tilde{A_n})\\
&=\text{det}_{HS}(\mathds{1}-\tilde{A})=\det(\mathds{1}-A).
\end{align*}
With the same arguments we obtain
\begin{align*}
\lim_{n\rightarrow\infty}\det(\mathds{1}-B_n)=\det(\mathds{1}-B).
\end{align*}
Because $A_n,B_n\in\mathcal{F}(X)$ we get for every $n\in\Natural$
\begin{align*}
|\det(\mathds{1}-A_n)-\det(\mathds{1}-B_n)|\leq \vectornorm{A_n-B_n}_\mathcal{N}\exp\bigg(\frac{1}{2}(\vectornorm{A_n}_\mathcal{N}+\vectornorm{B_n}_\mathcal{N}+1)^2\bigg)
\end{align*}
For $n\rightarrow\infty$ follows
\begin{align*}
|\det(\mathds{1}-A)-\det(\mathds{1}-B)|\leq \vectornorm{A-B}_\mathcal{N}\exp\bigg(\frac{1}{2}(\vectornorm{A}_\mathcal{N}+\vectornorm{B}_\mathcal{N}+1)^2\bigg).
\end{align*}
\end{proof}

Now we are ready to prove the main assertion.

\begin{theorem}\label{theorem3.5}
Let $\lambda \mapsto K(\lambda)$ be an analytic map in $\Omega \subseteq \mathbb{C}$ with values in $\mathcal{N}(X)$. Then $d$, defined by

\begin{align*}
d(\lambda):=\det(\mathds{1}-K(\lambda)),
\end{align*}

is a holomorphic function in $\Omega$
\end{theorem}

\begin{proof}

According to Lemma \ref{lemma3.1} for every $\lambda_0\in\Omega$ there is a sequence $\big(\lambda \mapsto K_{n,\lambda_0}(\lambda)\big)_{n\in\N}\subseteq \mathcal{F}(X)$ with the property 

\begin{align*}
\vectornorm{K(\lambda)-K_{n,\lambda_0}(\lambda)}_\mathcal{N}\rightarrow 0 \text{ locally uniformly on }B_{R_{\lambda_0}},
\end{align*}

and by Lemma \ref{lemma3.2}

\begin{align*}
d_{n,\lambda_0}:\lambda \mapsto \det(\mathds{1}-K_{n,\lambda_0}(\lambda))
\end{align*}

is holomorphic.
Now with Lemma \ref{lemma3.4} we get

\begin{align*}
|d_{n,{\lambda_0}}(\lambda)-d(\lambda)|\leq \vectornorm{K_{n,{\lambda_0}}(\lambda)-K(\lambda)}_\mathcal{N}\exp\bigg(\frac{1}{2}\bigg(\vectornorm{K_{n,{\lambda_0}}(\lambda)}_\mathcal{N}+\vectornorm{K(\lambda)}_\mathcal{N}+1\bigg)^2\bigg).
\end{align*}

There is a constant $c_{{\lambda_0}}>0$ such that
\begin{align*}
\exp\bigg(\frac{1}{2}\bigg(\vectornorm{K_{n,{\lambda_0}}(\lambda)}_\mathcal{N}+\vectornorm{K(\lambda)}_\mathcal{N}+1\bigg)^2\bigg)\leq c_{{\lambda_0}}
\end{align*}
locally uniformly on $B_{R_{\lambda_0}}$. The sequence of holomorphic functions $(d_{n,{\lambda_0}})$ converges locally uniformly to $d$ on $B_{R_{\lambda_0}}$. Hence for every $\lambda_0$ the function $d$ is on $B_{R_{\lambda_0}}$ the locally uniform limit of the holomorphic functions $d_{n,{\lambda_0}}$, and so by the Weierstrass convergence-theorem (see e.g. Remmert and Schumacher \cite{remmert}, p. 222) $d$ has to be holomorphic on $B_{R_{\lambda_0}}$ and hence also on $\Omega$.
\end{proof}

We can apply the general result in Theorem \ref{theorem3.5} to our initial problem of Section 2. 

\begin{theorem}\label{corollary3.7}
Let $Z_0\in\mathcal{B}(X)$, $X$ Banach space, and $Z=Z_0+K$, $K\in\NN(X)$. Then the determinant
\begin{align*}
d(\cdot)=\det\big(\mathds{1}-KR_{Z_0}(\cdot)\big)
\end{align*}
is holomorphic on $\Complex\setminus[a,b]$.\\
Therefore $h=d\circ\phi$, with $\phi$ from (\ref{8.1}) is a holomorphic function in $\mathbb{D}$.\\
Moreover there is the following connection between the algebraic multiplicity $m_\lambda(Z)$ of any eigenvalue $\lambda$ of $Z$ and the order $o_\lambda(d)$ of any zero of $d$.
\begin{align*}
\lambda \in\sigma_d(Z) \text{ with } m_\lambda(Z)=m \Leftrightarrow \lambda \in \mathbb{Z}(h) \text{ with } o_\lambda(d)=m.
\end{align*} 
\end{theorem}

\begin{proof}
Since $R_{Z_0}(\cdot)$ is analytic on $\rho(Z_0)$, $d$ is holomorphic on $\rho(Z_0)$.\\
In Section 2 we have already mentioned, that the zeros of $d$ coincide with the discrete spectrum of $Z$. So only we have to show that the algebraic multiplicity of any discrete eigenvalue of $Z$ is equal to its order as a zero of $d$.\\
For this lets fix an eigenvalue $\lambda_0 \in \sigma_{disc}(Z)$ and an $\epsilon>0$, such that \\ $B_{\epsilon}(\lambda_0)\cap \sigma_{disc}(Z)=\{\lambda_0\}$.\\
Next we choose a sequence $(K_n)\subseteq \mathcal{F}(X)$ with
\begin{align}
\|K-K_n\|\rightarrow 0 \text{ as }n\rightarrow \infty.\label{kompaktekonvergenz}
\end{align}
For $Z_n:=Z_0+K_n$ (\ref{kompaktekonvergenz}) implies
\begin{align}
\|Z-Z_n\|\rightarrow 0 \text{ as } n\rightarrow \infty. \label{zkonvergenz}
\end{align}
The statement in Gohberg, Goldberg and Kaashoek \cite{gohberg1} Chapter II Theorem 4.2. and (\ref{zkonvergenz}) implies that there is $N_\epsilon\in\mathbb{N}$ with
\begin{align}
\sum_{\mu\in\sigma_{disc}(Z_n)\cap B_\epsilon(\lambda_0)}m_{\mu}(Z_n)=m_{\lambda_0}(Z)\text{ for all } n\geq N_\epsilon. \label{vielfachheite}
\end{align}
Associated to $\{Z_n\}$ we have a sequence of holomorphic funtions
\begin{align*}
d_n(\lambda):=\det\big(\mathds{1}-K_nR_{Z_0}(\lambda)\big)\text{ for }\lambda\in\rho(Z_0)
\end{align*}
with $d_n(\lambda)=0$ iff $\lambda \in\sigma_{disc}(Z_n)$. Since $K_n\in\mathcal{F}(X)$ we know that $K_nR_{Z_0}(\lambda)\in\mathcal{F}(X)$. So for every $\lambda \in \rho(Z_0)$ we can consider $K_nR_{Z_0}(\lambda)$ as a Hilbert-Schmidt operator. Following Hansmann \cite{hansmann} p.20-22, we deduce
\begin{align}
\mu\in\sigma_{disc}(Z_n)\text{ with }m_\mu(Z_n)=m\Leftrightarrow d_n(\mu)=0\text{ and }o_{\mu}(d_n)=m.\label{muvielf}
\end{align}
Using (\ref{kompaktekonvergenz}) we can conclude that $\|KR_{Z_0}(\lambda)-K_nR_{Z_0}(\lambda)\|\rightarrow 0$ locally uniformly. This result implies that $d_n\rightarrow d$ locally uniformly. \\
Thus we can find $N\geq N_\epsilon$ such that 
\begin{align*}
|d_n(\lambda)-d(\lambda)|\leq d(\lambda)|
\end{align*}
$\text{ for all }\lambda \in \partial B_\epsilon(\lambda_0)\text{ and for }n\geq N$.
Rouche's Theorem (see e.g. \cite{rudin}, p.225) provides 
\begin{align*}
\sum_{\mu\in\mathcal{Z}(d_n)\cap B_\epsilon(\lambda_0)}o_{\mu}(d_n)=o_{\lambda_0}(d)\text{ for all }n\geq N.
\end{align*}
Now using this formula, the formula (\ref{vielfachheite}) and the equivalence (\ref{muvielf}) we receive
\begin{align*}
o_{\lambda_0}(d)=m_{\lambda_0}(Z).
\end{align*}
On the other hand, if $\lambda_0$ is a zero of $d$, we already know that $\lambda_0$ is a discrete eigenvalue of $Z$. Hence, by the previous argumentation, the algebraic multiplicity of $\lambda_0$ as an eigenvalue of $Z$ is equal to the order of $\lambda_0$ as a zero of $d$.
\end{proof}

Due to the first assertion in Theorem \ref{corollary3.7} Jensen's identity (see (\ref{8.2})) holds for $h=d\circ \phi$ and we can apply the theory developed in \cite{demuth}. In particular we can use the following result (see Hansmann and Katriel \cite{hansmannkatriel} Theorem 4), which is an extension of a theorem by Borichev, Golinskii
and Kupin \cite{borichev}.

\begin{theorem}\label{theorem3.7}
Let $h:\mathbb{D}\rightarrow \mathbb{C}$ be holomorphic and
\begin{align*}
|h(w)| \leq \exp\bigg(\frac{C_0|w|^\gamma}{(1-|w|)^\alpha\prod_{j=1}^N|w-\xi_j|^{\beta_j}}\bigg),\quad w \in\mathbb{D}, h(0)=1
\end{align*}
with $|\xi_j|=1, \xi_i\neq \xi_j,$ for $i\neq j,\alpha>0, \beta_j\geq 0, \gamma,C_0\geq 0$. Then we have for $\epsilon,\tau >0$
\begin{align*}
\sum_{w\in\mathcal{Z}(h)} \frac{(1-|w|)^{\alpha+\tau+1}}{|w|^{(\gamma-\epsilon)_+}}\prod_{j=1}^N|w-\xi_j|^{(\beta_j-1+\tau)_+} \leq C(\alpha,\beta,\xi,\epsilon,\tau)C_0,
\end{align*}
where $C(\alpha,\beta,\xi,\epsilon,\tau)>0$ denotes a constant depending on $\alpha,\beta,\xi,\epsilon,\tau$; 
$(x)_+:=\max(x,0)$ for $x\in\mathbb{R}$. 
\end{theorem}
This theorem will be very useful for the examples in the next sections. (see (\ref{resolventenabschaetzung}) and (\ref{eq5.1}))
\section{The discrete Laplacian on $\lpz$}

For illustration we will apply the results from Sections 2 and 3 to the discrete Laplacian on $\lpz$, $1\leq p\leq \infty$. This operator $\Delta_p : \lpz \rightarrow \lpz$ is given by 
\begin{align*}
(\Delta_p f)(n):= f(n-1)+f(n+1),\quad f\in \lpz.
\end{align*}
$\Delta_p$ is a bounded operator on $\lpz$, $p\in[0,\infty]$. It can be rewritten as 
\begin{align}
\Delta_pf=
\begin{pmatrix}
\ddots &\ddots& \ddots & & &\\
		& 1&0&1 &&\\
		&  &1&0&1& &\\
		&  & &1&0&1&\\ 
		&  & & &\ddots &\ddots &\ddots 
\end{pmatrix}
\begin{pmatrix}
\vdots\\
f(-1)\\
f(0)\\
f(1)\\
\vdots
\end{pmatrix}
\end{align}

for $f\in\lpz$, $f=\begin{pmatrix}
\vdots\\
f(-1)\\
f(0)\\
f(1)\\
\vdots
\end{pmatrix}$.\\
$\Delta_p$ is a Toeplitz-operator with the generating sequence $a_n=1$ for $n=1,-1$, $a_n=0$ else.\\
$\Delta_p\in\mathcal{B}(\lpz)$ follows by 
\begin{align*}
\vectornorm{\Delta_p f }_p \leq 2\vectornorm{f}_p, \, f\in l^p(\mathbb{Z}).
\end{align*}
For the essential spectrum $\sigma_{ess}(\Delta_p)$ we have (Duren \cite{duren}, p. 23)
\begin{align*}
\sigma_{ess}(\Delta_p) =\{\hat{a}(\theta):\theta \in [0,2\pi]\}
\end{align*}
where
\begin{align*}
\hat{a}(\theta):=\sum_{m=-\infty}^\infty a_m e^{-im\theta} = 1 \cdot e^{i\theta}+1\cdot e^{-i\theta}=2\cos\theta.
\end{align*}
That implies
\begin{align*}
\sigma_{ess}(\Delta_p)=[-2,2].
\end{align*}
Since the winding-number of $\hat{a}$ for every $z\notin[-2,2]$ is equal to 0, we can conclude (\cite{duren}, p.143)
\begin{align*}
\sigma(\Delta_p)=\sigma_{ess}(\Delta_p)=[-2,2].
\end{align*}
\begin{proposition}
The resolvent of $\Delta_p$ for $z\in\rho(\Delta_p)$ is given by
\begin{align}
R_{\Delta_p}(z)=(z\mathds{1}-\Delta_p)^{-1}:=
\begin{pmatrix}
	  & \vdots 	 & \vdots    &	\vdots & \vdots  &	\\
\dots & b_{-1}(z)& b_{0}(z) & b_{1}(z) & b_{2}(z)& \dots \\
\dots & b_{-2}(z)& b_{-1}(z)& b_{0}(z) & b_{1}(z)& \dots \\
\dots & b_{-3}(z)& b_{-2}(z)& b_{-1}(z)& b_{0}(z)& \dots \\
	  & \vdots   & \vdots    & \vdots  & \vdots   &		
\end{pmatrix}, \label{resol}
\end{align}
with 
\begin{align*}
b_k(z):=\bigg( \frac{z\pm \sqrt{z^2-4}}{2}\bigg)^{|k|}\frac{1}{\sqrt{z^2-4}} \text{ for }k\in\mathbb{Z} \text{ and } z\in\rho(\Delta_p)=\mathbb{C}\setminus[-2,2].
\end{align*}
The sign of $\sqrt{z^2-4}$ should be chosen, such that the inequality $|z\pm\sqrt{z^2-4}|<2$ is fulfilled.\\
Moreover, $R_{\Delta_p}(z)$ is a bounded operator and 
\begin{align*}
\vectornorm{R_{\Delta_p}(z)}_{l^p}\leq \frac{1}{|z^2-4|^{1/2}}\frac{2+|z\pm \sqrt{z^2-4}|}{2-|z\pm\sqrt{z^2-4}|},
\end{align*}
$z\in\rho(\Delta_p)=\Complex \setminus[-2,2]$.
\end{proposition}
\begin{proof}
Let $B_p(z)$ with $z\in\rho(\Delta_p)$ the right hand side of (\ref{resol}), then we have (see e.g. Kato \cite{kato}, p. 143)
\begin{align*}
\vectornorm{B_p(z)f}_p&\leq \bigg(\sum_{k=-\infty}^\infty |b_k(z)|\bigg)^{1-\frac{1}{p}}\bigg(\sum_{k=-\infty}^\infty |b_k(z)|\bigg)^\frac{1}{p}\vectornorm{f}_p\\
&=\bigg(\sum_{k=-\infty}^\infty |b_k(z)|\bigg)\vectornorm{f}_p=\frac{1}{|z^2-4|^\frac{1}{2}}\bigg(2\frac{1}{1-\left|\frac{z\pm\sqrt{z^2-4}}{2}\right|}-1\bigg) \vectornorm{f}_p\\
&=\frac{1}{|z^2-4|^\frac{1}{2}}\bigg(\frac{2+|z\pm\sqrt{z^2-4}|}{2-|z\pm \sqrt{z^2-4}|}\bigg)\vectornorm{f}_p.
\end{align*}
And a direct calculation shows 
\begin{align*}
B_p(z)(\Delta_p-z)f=f=(\Delta_p-z)B_p(z)f
\end{align*}
such that $B_p(z)=R_{\Delta_p}(z)$.
\end{proof}
$\Delta_p$ plays the role of $Z_0$ in the sections above. Now we add a nuclear perturbation and study the discrete spectrum of the perturbed operator.\\
Let $K\in\NN(\lpz)$ and denote
\begin{align*}
Z=\Delta_p+K.
\end{align*}
Since $K$ is compact we have
\begin{align*}
\sigma_{ess}(Z)=\sigma_{ess}(\Delta_p)=[-2,2].
\end{align*}
According to Corollary \ref{corollary3.7} a holomorphic function is given by
\begin{align*}
d(\lambda)=\det\big(\mathds{1}-KR_{\Delta_p}(z)\big),
\end{align*}
$d$ defined on $\Complex\setminus[-2,2]$. In order to use the method from Section 2 we take the conformal map (see (\ref{8.1}))
\begin{align*}
\phi(w)=w+w^{-1}.
\end{align*}
$\phi$ maps $\mathbb{D}\setminus \{0\}$ to $\Complex \setminus[-2,2]$.\\
Denote $h:=d\circ \phi$, then (see (\ref{9.4}))
\begin{align*}
\log|h(w)|\leq \frac{1}{2}\vectornorm{K}_{\NN}^2\vectornorm{R_{\Delta_p}(\phi(w))}^2.
\end{align*}
For the norm of the resolvent we obtain
\begin{align*}
\vectornorm{R_{\Delta_p}(w+w^{-1})}&\leq \frac{1}{|(w+w^{-1})^2-4|^\frac{1}{2}}\bigg(\frac{2+|w+w^{-1}\pm\sqrt{(w+w^{-1})^2-4}|}{2-|w+w^{-1}\pm\sqrt{(w+w^{-1})^2-4}|}\bigg)\\
&=\frac{1}{|w-w^{-1}|}\bigg(\frac{2+|w+w^{-1}\pm\sqrt{(w-w^{-1})^2}|}{2-|w+w^{-1}\pm\sqrt{(w-w^{-1})^2}|}\bigg) \\
&=\frac{|w|}{|w^2-1|}\bigg(\frac{2+2|w|}{2-2|w|}\bigg)\\
&\leq \frac{2|w|}{|w-1||w+1|(1-|w|)}, \quad w\in\mathbb{D}\setminus\{0\}.
\end{align*}
Hence 
\begin{align}
\log|h(w)|\leq 2 \|K\|_{\NN}^2\frac{|w|^2}{(1-|w|)^2|w-1|^2|w+1|^2}.\label{resolventenabschaetzung}
\end{align}
There is a holomorphic extension for $h$ to $\mathbb{D}$ realized by $h(0):=d(\infty):=1$. Using Theorem \ref{theorem3.7} with $\epsilon=1-\tau$
\begin{align*}
\sum_{w\in\mathcal{Z}(h)}\frac{(1-|w|)^{3+\tau}}{|w|^{1+\tau}}|w^2-1|^{1+\tau}\leq C(\tau)\|K\|_{\NN}^2
\end{align*}
with $0<\tau<1$.

For transforming these estimate to an estimate for $\sigma_{disc}(Z)$ we use the following relations (\cite{demuth}, p. 130).
\begin{lemma}\label{lemma4.2}
Let $z=w+w^{-1}$, $w\in\mathbb{D}\setminus\{0\}$. Then we have
\begin{align*}
\frac{1}{2}\frac{|w^2-1|(1-|w|)}{|w|}\leq \textnormal{dist}(z,[-2,2])\leq \frac{1+\sqrt{2}}{2}\frac{|w^2-1|(1-|w|)}{|w|}
\end{align*}
and 
\begin{align*}
{\left|\frac{w^2-1}{w}\right|}^2=|z^2-4|.
\end{align*}
\end{lemma}

\begin{theorem}\label{theorem4.3}
Let $Z=\Delta_p+K$ be in $\lpz$ with $K\in\NN(\lpz)$, $1\leq p\leq \infty$. 
Then we get for $\tau>0$
\begin{align}
\sum_{z\in\sigma_{disc}(Z)}\frac{\textnormal{dist}(z,[-2,2])^{3+\tau}}{|z^2-4|}\leq C(\tau)\|K\|_{\NN}. \label{lti}
\end{align}
\end{theorem}

\begin{proof}
Let $w\in\mathbb{D}\setminus\{0\}$, $z=w+w^{-1}$. By Lemma \ref{lemma4.2} we obtain
\begin{align*}
(1-|w|)^{3+\tau}\left|\frac{w^2-1}{w}\right|^{1+\tau}&=\bigg(\frac{(1-|w|)|w^2-1|}{|w|}\bigg)^{3+\tau}\left|\frac{w}{w^2-1}\right|^2\\
&\geq \bigg(\frac{2}{1+\sqrt{2}}\bigg)^{3+\tau}\frac{\textnormal{dist}(z,[-2,2])^{3+\tau}}{|z^2-4|}.
\end{align*}
\end{proof}

\begin{remark}
Whenever $1\leq p < \infty$ the Banach space $l^p(\Integer)$ is compatible to $l^2(\Integer)$. Having a nuclear operator $K_1$ on $l^p(\Integer)$ and a Hilbert-Schmidt operator $K_2$ on $l^2(\Integer)$ which are conistent we know (see Davies \cite{davies} p. 107) that the spectra of $K_1$ and $K_2$ coincide. Moreover in this case $K_1R_{\Delta_p}(z)$ is nuclear in $l^p(\Integer)$, $K_2R_{\Delta_2}(z)$ is Hilbert-Schmidt in $l^2(\Integer)$, both operators are consistent and $1\in \sigma(K_1R_{\Delta_p}(z))$ iff $1\in\sigma(K_2R_{\Delta_2}(z))$ and this means that the zero set of $\det(1-K_1R_{\Delta_p}(z))$ coincides with the zero set of $\det(1-K_2R_{\Delta_2}(z))$ (compare with the discussion after Lemma \ref{2.5}). That means we only have to consider the Hilbert space case which was already studied by Borichev, Golinskii, Kupin \cite{borichev} and Hansmann, Katriel \cite{hansmannkatriel}.
Recall that $l^\infty(\Integer)$ and $l^2(\Integer)$ are not compatible, and the assertion in Theorem \ref{theorem4.3} holds also for $p=\infty$.
Nevertheless there are nuclear operators on $l^p(\Integer)$ which are not consistent to a Hilbert-Schmidt operator in $l^2(\Integer)$ (see Example \ref{nuclearnothibert}).\\

\end{remark}

\begin{remark}
If $X$ is the Hilbert-space $l^2(\mathbb{Z})$ we have $\NN(X)=\mathcal{S}_1(l^2(\mathcal{Z}))$ (trace class operators). With Proposition \ref{theorem4.3} we only get that 
\begin{align*}
\sum_{z\in\sigma_{disc}(Z)}\frac{\textnormal{dist}(z,\sigma_{ess}(Z))^{3+\tau}}{|z^2-4|}\leq KC(\tau).
\end{align*}
However for Jacobi-operators in $l^2(\mathbb{Z})$ there is a better estimate given in \cite{demuth} p. 149
\begin{align*}
\sum_{z\in\sigma_{disc}(Z)}\frac{\text{dist}(z,\sigma_{ess}(Z))^{1+\tau}}{|z^2-4|^{\frac{1}{2}+\frac{\tau}{4}}}<\infty.
\end{align*}
\end{remark}
\begin{remark}
With (\ref{lti}) it is possible to give an estimate for the number of eigenvalues in certain regions of the complex plane.\\
Define
\begin{align*}
M:=\{z\in\Complex : \text{Re}(z)<-2, r<|z+2|<R\}
\end{align*}
with $R>r>0$ (since the operator norm is a bound for the spectrum, it makes sense to set $\|Z\|-2\geq R$). Then for $\lambda \in M$ the inequalities 
\begin{align*}
\text{dist}(\lambda,[-2,2])>r, \, \frac{1}{|\lambda+2|}>\frac{1}{R}, \, \frac{1}{|\lambda-2|}>\frac{1}{R+4}
\end{align*}
are valid.\\
By (\ref{lti}) we obtain for every $\tau>0$
\begin{align*}
&\sum_{\lambda\in\sigma_{disc}(Z)\cap M}\frac{r^{3+\tau}}{R(R+4)}\\
\leq &\sum_{z\in\sigma_{disc}(Z)\cap M}\frac{\textnormal{dist}(z,\sigma_{ess}(Z))^{3+\tau}}{|z^2-4|}\\
\leq &\sum_{z\in\sigma_{disc}(Z)}\frac{\textnormal{dist}(z,\sigma_{ess}(Z))^{3+\tau}}{|z^2-4|}\leq \|K\|_{\NN}^2C(\tau).
\end{align*}
We can conclude 
\begin{align*}
\#\big(\sigma_{disc}(Z)\cap M\big)\leq \frac{R(R+4)}{r^{3+\tau}}C(\tau)\|K\|_{\NN}^2.
\end{align*}
The same estimate is valid for eigenvalues in $M:=\{z\in\Complex : \text{Re}z>2, r<|z-2|<R\}$.\\
Defining
\begin{align*}
N:=\{z\in\Complex: -2\leq \text{Re}(z)\leq 2,\, r<|\text{Im}z|<R\}
\end{align*}
\begin{figure}
\centering
\begin{minipage}{6cm}
\includegraphics[height=4.5cm]{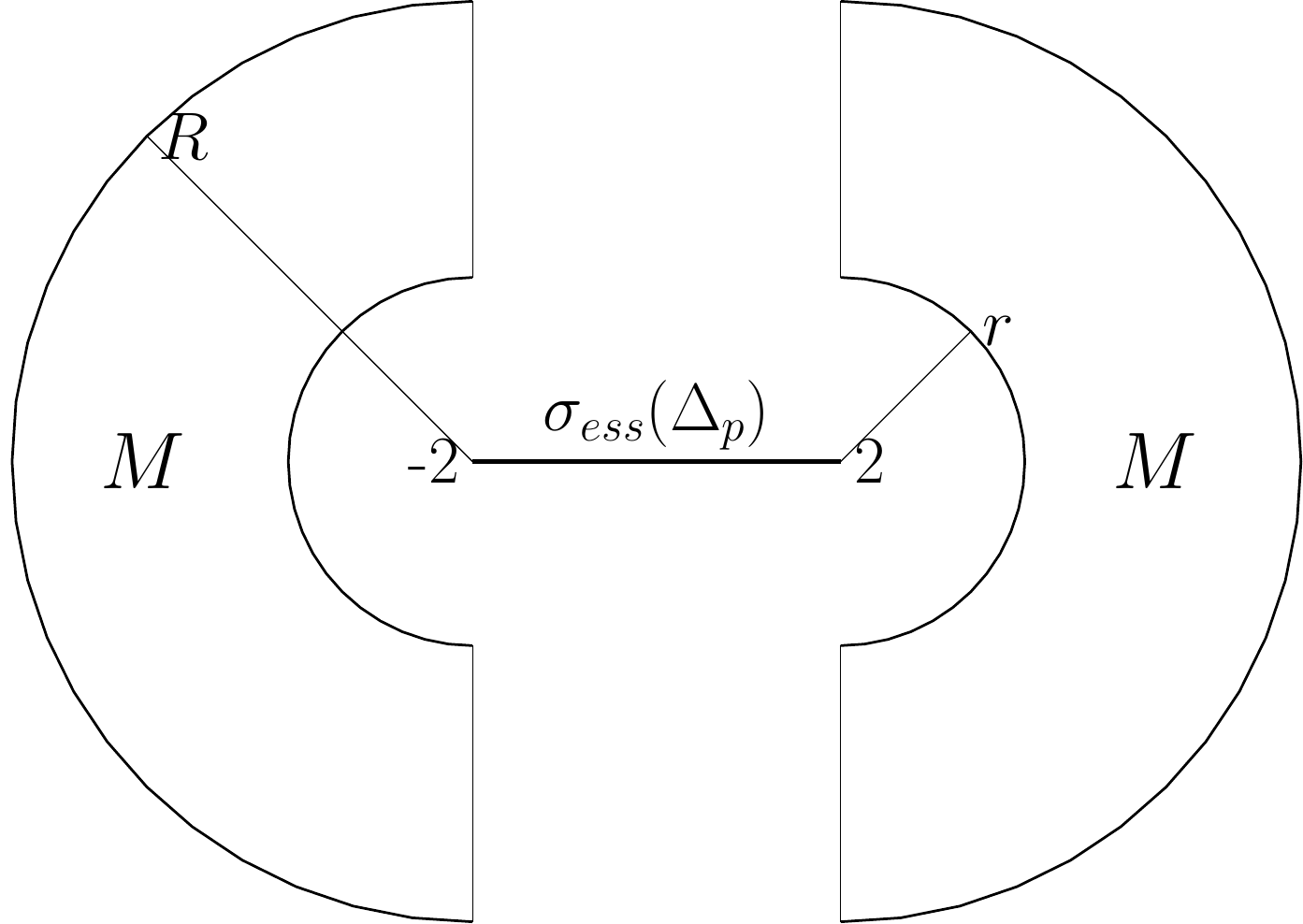}
\caption{Region $M$ with \newline $\|Z\|-2>R>r>2$.}
\end{minipage}
\begin{minipage}{6cm}
\includegraphics[height=4.5cm]{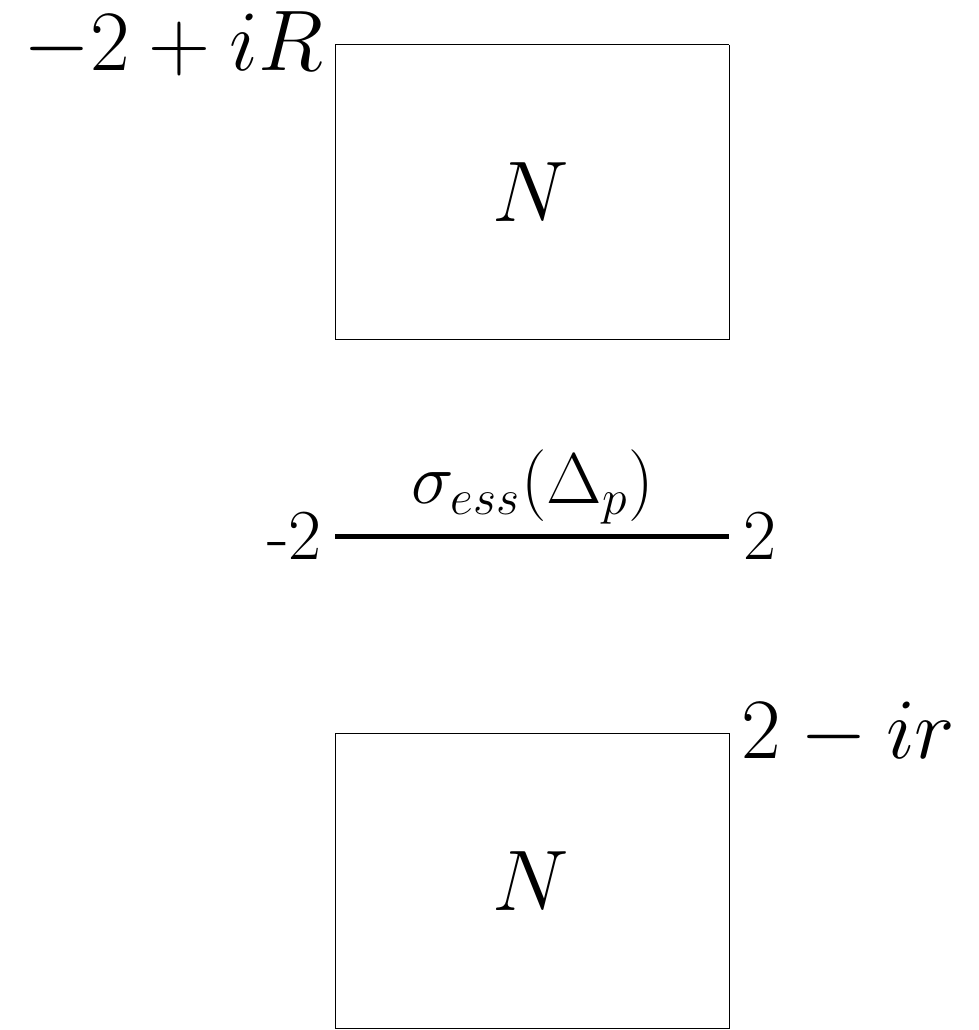}
\caption{Region $N$ with \newline $\|Z\|\geq R>r>0$.}
\end{minipage}
\end{figure}
with $\|Z\|\geq R>r>0$ we receive for every $\lambda \in N$ the inequalities
\begin{align*}
\text{dist}(\lambda,[-2,2])>r,\, \frac{1}{|\lambda+2|}>\frac{1}{\sqrt{16+R^2}},\, \frac{1}{|\lambda-2|}>\frac{1}{\sqrt{16+R^2}}.
\end{align*}
By the same argumentation we get for every $\tau>0$
\begin{align*}
\sum_{\lambda\in\sigma_{disc}(Z)\cap N}\frac{r^{3+\tau}}{16+R^2}\leq C(\tau)\|K\|_{\NN}.
\end{align*}
Hence for the number of eigenvalues in $N$
\begin{align*}
\#\big(\sigma_{disc}(Z)\cap N\big)\leq \frac{16+R^2}{r^{3+\tau}}\|K\|_{\NN}^2C(\tau).
\end{align*}

\end{remark}
\begin{remark}
If we take a sequence $\{\lambda_k\}_{k\in\Natural}\subseteq \sigma_{disc}(Z)$ converging to some $\lambda_0\in[-2,2]$ we can consider the following cases (restricting the sequence to a subsequence):\\
\qquad (i.a) $\lambda_0 =-2$ and Re($\lambda_k)\leq -2$. \qquad (i.b) $\lambda_0=2$ and Re($\lambda_k)\geq 2$.\\
\qquad (ii.a) $\lambda_0 =-2$ and Re($\lambda_k)> -2$. \qquad (ii.b) $\lambda_0=2$ and Re($\lambda_k)< 2$.\\
\qquad (iii) $\lambda_0 \in(-2,2)$.
\\Because of the symmetry it is sufficient to consider only the cases (i.a), (ii.a) and (iii). \\
In case (i.a) the estimate (\ref{lti}) implies for $\tau>0$. 
\begin{align*}
\sum_{k=1}^\infty |\lambda_k+2|^{2+\tau}<\infty,
\end{align*}
in (ii.a)
\begin{align*}
\sum_{k=1}^\infty \frac{|\text{Im}(\lambda_k)|^{3+\tau}}{|\lambda_k+2|}<\infty
\end{align*}
and finally in (iii)
\begin{align*}
\sum_{k=1}^\infty |\text{Im}(\lambda_k)|^{3+\tau}<\infty.
\end{align*}
\end{remark}

\section{Nuclear perturbations of the multiplication operator on $C[\alpha,\beta]$}

As another easy application, consider $X=C[\alpha,\beta]$ (the space of continuous functions). We define 
\begin{align*}
Z_0:X\rightarrow X \text{ with } (Z_0f)(t):=M(t)f(t),
\end{align*}
where $M$ is a real-valued contiuous function on $[\alpha,\beta]$. \\
We know $Z_0\in\mathcal{B}(X), \sigma(Z_0)=\sigma_{ess}(Z_0)=[\min(M),\max(M)]=:[a,b]$ and 
\begin{align*}
\big(R_{Z_0}(\lambda)f\big)(x)=\frac{f(x)}{M(x)-\lambda},\, \lambda\in\rho(Z_0).
\end{align*}
In this example it is possible to compute the operator norm of the resolvent exactly and we receive
\begin{align*}
\|R_{Z_0}(\lambda)\|=\|\frac{1}{M-\lambda}\|_\infty=\frac{1}{\text{dist}(\lambda,[a,b])}.
\end{align*}
Defining an integral operator 
\begin{align*}
K:X\rightarrow X\text{ with } Kf(t):=\int_\alpha^\beta k(t,s)f(s)ds
\end{align*}
with $k\in C[\alpha,\beta]^2$ we know by Example \ref{beispiel} that this operator is nuclear.\\
Then for 
\begin{align*}
Z:=Z_0+K
\end{align*}
the function
\begin{align*}
d(\lambda):=\det\big(1-R_{Z_0}(\lambda)K\big), \, \lambda \in \rho(Z_0)
\end{align*}
defines a holomorphic function with zero-set equal to $\sigma_{disc}(Z)$ and 
\begin{align*}
|d(\lambda)|\leq \frac{1}{2}\|K\|_\mathcal{N}^2\frac{1}{\text{dist} {(\lambda,[a,b])^2}}.
\end{align*}
Setting $\phi(w):=\frac{b-a}{4}(w+w^{-1}+2),\, w\in\mathbb{D}\setminus\{0\}$ ($\phi$ maps $\mathbb{D}\setminus \{0\}$ to $\mathbb{C}\setminus[a,b]$) we receive (for the holomorphic function $d\circ \phi$)
\begin{align*}
|(d\circ \phi) (w)|\leq \frac{1}{2}\|K\|_\mathcal{N}^2\frac{1}{\text{dist}(\phi(w),[a,b])^2} .
\end{align*}
Using the estimate 
\begin{align*}
\frac{b-a}{8}\frac{|w^2-1|(1-|w|)}{|w|}\leq \text{dist}(\phi(w),[a,b])\leq \frac{(b-a)(1+\sqrt{2})}{8}\frac{|w^2-1|(1-|w|)}{|w|}
\end{align*}
with  $w\in\mathbb{D}\setminus\{0\}$,
which is a generalization of the estimate in Lemma \ref{lemma4.2} (see \cite{demuth}, Lemma 4.2.1) we obtain
\begin{align}
|d\circ\phi(w)|\leq \frac{1}{2}C(a,b)\|K\|_\mathcal{N}^2\frac{|w|^2}{|w^2-1|^2(1-|w|)^2} \label{eq5.1}
\end{align}
and so by Theorem \ref{theorem3.7}
\begin{align*}
\sum_{w\in\mathcal{Z}(d\circ\phi)}\frac{(1-|w|)^{3+\tau}}{|w|^{1+\tau}}|w^2-1|^{1+\tau}\leq C(\tau,a,b)\|K\|_\mathcal{N}^2.
\end{align*}
In analogy to the proof of Theorem \ref{theorem4.3} we can deduce:
\begin{theorem} Let $Z=Z_0+K$ defined as described above, then 
\begin{align*}
\sum_{\lambda\in\sigma_{disc}(Z)}\frac{\text{dist}(\lambda,[a,b])^{3+\tau}}{|\lambda-a||\lambda-b|}\leq C(\tau,a,b)\|K\|_\mathcal{N}.
\end{align*}
\end{theorem}

\newpage
\section*{Acknowledgement}
The idea to use this complex analysis method for studying the discrete spectrum of operators in Banach spaces perturbed by nuclear operators goes back to discussions with Guy Katriel and Marcel Hansmann. Their critical reading of the manuscript has improved considerably the article. We are grateful to both of them. Also we thank Friedrich Philipp for very valuable discussions.

\end{document}